\documentclass[12pt]{article}
\usepackage{amsfonts}
\usepackage{amsmath}
\usepackage{mathrsfs}
\usepackage{amssymb}
\usepackage{color,amscd,array,graphicx,mathdots,bm,amsfonts,epsfig,amsmath}
\usepackage{amsthm}
\usepackage{enumerate}
\usepackage{geometry}
\renewcommand{\theequation}{\thesection.\arabic{equation}}
\newtheorem{theorem}{Theorem}[section]

\newtheorem{definition}{Definition}[section]
\newtheorem{lemma}{Lemma}[section]
\newtheorem{proposition}{Proposition}[section]

\newcommand{\levy}{L\'{e}vy}
\newcommand{\seminaire}{S\'{e}minaire}
\newcommand{\probabilites}{Probabilit\'{e}s}
\newcommand{\poincare}{Poincar\'{e}}
\newcommand{\equations}{\'{e}quations}
\newcommand{\differrentielles}{diff\'{e}rentielles}

\makeatletter
\newcommand{\figcaption}{\def\@captype{figure}\caption}
\newcommand{\tabcaption}{\def\@captype{table}\caption}
\makeatother

\def\de{{\delta}}
\def\ep{{\epsilon}}
\def\la{{\lambda}}\def\om{{\omega}}\def\si{{\sigma}}

\def\Ga{{\Gamma}}

\newfam\msbmfam\font\tenmsbm=msbm10\textfont
\msbmfam=\tenmsbm\font\sevenmsbm=msbm7
\scriptfont\msbmfam=\sevenmsbm
\def\BB{\mathbb B}
\def\EE{\mathbb E}\def\PP{\mathbb P}
\def\RR{\mathbb R}

\def\cB{{\cal B}}
\def\cD{{\cal D}}

\def\cR{{\cal R}}

\def\<{\left<}\def\>{\right>}

\def\({\left(}\def\){\right)}

\begin{document}
\title{Unique strong solutions of {\levy} processes driven stochastic differential equations with discontinuous coefficients \footnote{JX's research is supported by Southern University of Science and Technology Start up fund Y01286220 and National Science Foundation of China grants  61873325 and 11831010., JZ's research is supported by Macao Science and Technology Development Fund FDCT 025/2016/A1 and NSERC (249554-2011)
 and XZ's research is supported by NSERC (RGPIN-2016-06704).}
}
\author{Jie Xiong\footnote{Department of Mathematics, Southern University of Science and Technology, Shenzhen, China}, Jiayu Zheng\footnote{School of Mathematics (Zhuhai), Sun Yat-sen University, Zhuhai 519082, Guangdong, China} and
 Xiaowen Zhou\footnote{Department of Mathematics and Statistics, Concordia
University, Canada.} }
\date{}
\maketitle
\begin{abstract}
We study the strong solutions for a class of one-dimensional stochastic differential  equations driven by a Brownian motion and a pure jump {\levy} process.  Under fairly general conditions on the coefficients, we prove the pathwise uniqueness  by showing the weak uniqueness and applying a local time technique.\bigskip
\end{abstract}


\section{Introduction}
\setcounter{equation}{0}
\renewcommand{\theequation}{\thesection.\arabic{equation}}

Suppose that $U$ is a complete separable metric space and that $\mu$ is a $\sigma$-finite Borel measure on $U$. Suppose that  $(\sigma(x),b(x),g(x,\cdot))$ is a $\RR^2\times L^1(U,\mu)$-valued  bounded Borel function of $x\in\mathbb{R}$ with at most countably many discontinuity points. This condition, referred as Condition (A), will be assumed throughout this article.

 For a filtered probability space $(\Omega, \mathscr{F}, (\mathscr{F}_t), \mathbb{P})$,
let $B\equiv(B_t)$ be a standard $(\mathscr{F}_t)$-Brownian motion
 and  $(p_t)$ be an independent $(\mathscr{F}_t)$-Poisson point process on $U$  with intensity measure $\mu$.

Let $N(ds, du)$  be the Poisson random measure on $\mathbb{R}_{+}\times U$ associated with  $(p_t)$. In this paper, we study the following stochastic differential equation:
\begin{align}\label{eq1}
X_t = X_0 + \int_0^t b(X_{s}) ds + \int_0^t \sigma(X_{s}) dB_s + \int_0^t \int_{U} g(X_{s-}, u) N(ds,du).
\end{align}

The question of pathwise uniqueness for one-dimensional stochastic differential equations with non-Lipschitz coefficients driven by one-dimensional Brownian motion has been resolved in 1971 by Yamada and Watanabe \cite{YW}. Recently, stochastic differential equations of jump type attracts a lot of  attention. Komatsu \cite{TK} and Bass \cite{RF} showed  that the following stochastic differential equation
\begin{eqnarray}\label{eqf123}
d X_t = F(X_{t-})d L_t, \quad  t\ge0
\end{eqnarray}
admits a strong solution and satisfies pathwise uniqueness if $L\equiv (L_t)$ is a symmetric stable process with index $\alpha \in (1,2)$, and if $x \to F(x)$ is a bounded function with modulus of continuity $z \to \rho(z)$ satisfying
\begin{eqnarray}\label{eqf1234}
\int_{0+} \frac{1}{\rho(z)^{\alpha}}dz = \infty.
\end{eqnarray}
When the integral in (\ref{eqf1234}) is finite, Bass $et\;al$ \cite{BB}  constructed a continuous function $x \to \Phi(x)$ having continuity modulus $x \to \rho(x)$ for which the pathwise uniqueness for (\ref{eqf123}) fails. Recall that  the sample paths of $L$ are of bounded variations for $\alpha \in (0, 1)$ and of unbounded variations for $\alpha \in [1, 2)$.

Under Lipschitz conditions, the existence and uniqueness of strong solutions of jump-type stochastic equations can be established by arguments based on Gronwall’s inequality and the results on continuous-type equations; see e.g. Ikeda and Watanabe \cite{IW}. Moreover, the pathwise uniqueness for SDEs with H\"older continuous diffusion coefficients has been extensively studied (see the works of Fournier \cite{NF} and Li and Mytnik \cite{LM}). However, to the best of our knowledge, there are few results about the pathwise uniqueness for one dimensional  SDE (1.1) with discontinuous coefficients.

The local time technique was firstly introduced by Le Gall \cite{LG} to study the pathwise uniqueness of classical SDEs without jumps. Then, the so-called (LT) condition (see the Definition 1.2 for further details) was introduced by Barlow and Perkins \cite{BP} as another tool to prove the pathwise uniqueness. This was also used by Le Gall \cite{LG2} to study stochastic equations involving local times. The importance of the (LT) condition lies in the following observation: if the diffusion coefficient $\sigma$ satisfies the (LT) condition, and $X^1$ and $X^2$ are two solutions, then so is $X^1\vee X^2$, which follows immediately from Tanaka's formula. Therefore, if the weak uniqueness holds, then the (LT) condition implies the pathwise uniqueness.

It is difficult to prove the pathwise uniqueness of (\ref{eq1}) when the coefficients are discontinuous. In the present paper, we first consider the weak existence and weak uniqueness of (\ref{eq1}) for which the conditions on coefficients could be weakened substantially. Then under the (LT) condition, we prove the pathwise uniqueness.

Before we give rigorous statements of our main results in Section 2, we recall some definitions.
	
	\begin{definition}
		\begin{enumerate} [(i)]
			\item

			A weak solution of (\ref{eq1}) is a triple $(X, W, N)$ on a filtered probability space $(\Omega, \mathcal{F},\{ \mathcal{F}_t\}_{t \ge 0}, P)$ such that $X_t$ is adapted to $\mathcal{F}_t$, $W_t$ is an $\{ \mathcal{F}_t\}_{t \ge 0}$-Brownian motion, $N$ is an $\{ \mathcal{F}_t\}_{t \ge 0}$-Poisson random measure, and $(X, W, N)$ satisfies (\ref{eq1}).

			\item We say that \textbf{weak uniqueness} holds for equation (\ref{eq1}) if, for any two weak solutions $(X, W, N), (\Omega, \mathcal{F}, \{ \mathcal{F}_t\}_{t \ge 0}, P)$ and $(\tilde{X}, \tilde{W}, \tilde{N}), (\tilde{\Omega}, \tilde{F}, \{ \tilde{\mathcal{F}}_t\}_{t \ge 0}, \tilde{P})$, with the same initial distribution, i.e., $\mathcal{L}(X_0) = \mathcal{L}(\tilde{X}_0)$, then $\mathcal{L}(X) = \mathcal{L}(\tilde{X})$.
			
			\item  \textbf{Pathwise uniqueness} is said to hold for (\ref{eq1}) if whenever $(X, W, N),$ $(\Omega, \mathcal{F}, \{ \mathcal{F}_t\}_{t \ge 0}, P)$ and $(\tilde{X}, W, N)$, $(\Omega, \mathcal{F}, \{ \tilde{\mathcal{F}}_t\}_{t \ge 0}, P)$ are weak solutions to (\ref{eq1}) with common Brownian motion $W$, common Poisson random measure $N$ (relative to possibly different filtrations) on a common probability space $(\Omega, \mathcal{F}, P)$ and with common initial value, i.e. $P(X_0 = \tilde{X}_0) = 1$, then $P(X_t = \tilde{X}_t \ \ \text{for all} \ \ t \ge 0) = 1.$
		\end{enumerate}
	\end{definition}

Now, we introduce  the (LT) condition, which will help us to get the pathwise uniqueness of SDE (\ref{eq1}).

\begin{definition}
	We say that  SDE (\ref{eq1}) satisfies (LT) condition
	if for any two solutions $X^1$ and $X^2$ of SDE (\ref{eq1}), we have $L^0_t (X^1 - X^2) = 0$ for all $t\geq 0$,
	where $L^0_t(X)$ is the local time of the semimartingale $X$ spent at location 0 up to time $t$.
\end{definition}

We refer the reader to  Section 4.7 of Protter \cite{PE} for more details about local times.

The rest of this paper is arranged as followed: In Section 2 we prove the existence of a weak solution by a martingale approach. In Section 3, we obtain the weak uniqueness by verifying the separating condition of Kurtz and Ocone \cite{KO}. Pathwise uniqueness is then proved in Section 4 under the (LT) condition. Some sufficient conditions for the (LT) condition are then presented. Finally, an example which motivates our research is discussed in Section 5. Throughout this paper, we will use $K$ to denote a constant whose value can  change from place to place.

\section{Weak existence}
\setcounter{equation}{0}
\renewcommand{\theequation}{\thesection.\arabic{equation}}

In this section, we study the existence of weak solutions to (\ref{eq1}) under conditions (2.a) and (2.b) below. We  first define an approximating sequence and prove its tightness. Then, to characterize the limit, we prove that the limit can not spend too much time at the points at which the coefficients are not continuous. By taking a limit we get a weak solution to SDE (\ref{eq1}).

We begin with  the following conditions:
\begin{itemize}
\item[(2.a)] There is a constant $K \ge 0$ such that
\begin{eqnarray}
 b(x)^2 + \sigma(x)^2  + \int_{U}\left[|g(x, u)| \vee g(x, u)^2 \right]\mu(du)  \le K, \quad \forall x \in \mathbb{R}. \nonumber
\end{eqnarray}
\item[(2.b)] There is a constant $\sigma_0 > 0$ such that
\[
| \sigma(x) | \ge \sigma_0, \quad  \forall x \in \mathbb{R}.
\]
\end{itemize}

The space of all bounded functions on $\mathbb{R}$ is denoted by $\mathbb{B}$. Let $A$ be the operator on $\mathbb{B}$ with domain $\cD(A):=\{f\in C^2(\RR)\cap \BB:\; f',\,f''\in\BB\}$ and for $f\in \cD(A)$,
\[Af(x):=b(x)f'(x)+\frac12\si^2(x)f''(x)
+\int_U\(f(x+g(x,u))-f(x)\)\mu(du).\]

Recall that a measurable stochastic process $X$ is a solution
of the $A$-martingale problem for generator $A$ if there exists a filtration
$\{\mathscr{F}_t\}$ such that
\[
f(X_t) - \int_0^t A f(X_t) ds
\]
is an $\{\mathscr{F}_t\}$-martingale for each $f \in \cD(A)$.

It is well known that the process $X$ solves the martingale problem if and only if it is a weak solution to SDE (\ref{eq1}). Therefore, for the weak existence we look for a process $X$ which solves the $A$- martingale problem.

We first construct a sequence of smooth functions to approximate the functions $b,\;\si,\;g$ which are not necessarily  continuous.
Let
\begin{eqnarray}
b_n(x) = \mathbb{E} \left(b(x +\xi_n)\right),  \nonumber
\end{eqnarray}
where $\xi_n$ is a normal random variable with mean 0 and variance $\frac1n$. Let $\sigma_n$ and $g_n$ be defined similarly.

By condition (2.a), it is easy to check that
\begin{eqnarray}
 b_n(x)^2 + \sigma_n(x)^2   + \int_{U}\left[|g_{n}(x, u)| \vee g_{n}(x, u)^2 \right]\mu(du)  \le K, \quad \forall x \in \mathbb{R}.  \nonumber
\end{eqnarray}
For each $n\ge 1$, the functions $(\sigma_n(x), b_n(x), g_{n}(x,\cdot))$, taking values in $\mathbb{R}^2\times (L^2(U,\mu)\cap L^1(U,\mu))$, are Lipschitz continuous in $x$. Then for every $n \ge 1$, by a well-known result on stochastic equations (see \cite{FL}, Theorem 2.5), there is a unique strong solution to
\begin{align}\label{eqn}
X_t^n = X_0 +\int_0^t b_n(X_{s}^n)ds+\int_0^t \sigma_n(X_{s}^n)dB_s
+ \int_0^t \int_{U} g_{n}(X_{s-}^n, u) N(ds,du).
\end{align}

To take the limit we first show a tightness result.

\begin{lemma}\label{lem2.1} Under condition (2.a),
for any solution ${X^n}$ of (\ref{eqn}) and $t>0$ we have
\[
\mathbb{E} \left[ \sup \limits_{0\le s<t}|X^n_s|^2 \right] \le 5\mathbb{E}|X_0|^2 + 5Kt^2 + 20Kt +5K^2t^2.
\]
\end{lemma}
\begin{proof}
The estimate can be obtained  directly by applying Doob's inequality to the martingale part of $X^n_t$ and by the boundedness of the coefficients, and hence, we omit the details.
\end{proof}

\begin{lemma}\label{lem_tight} Under condition (2.a), the sequence $\{X^n\}$ is tight in the Skorohod space $D([0,\infty),\mathbb{R})$.
\end{lemma}
\begin{proof}
 Constructing the solution under probability measure $\mathbb{Q}$ given by $\frac{d\mathbb{Q}}{d{\mathbb{P}}}=K^{-1}e^{-|X_0|}$ if necessary, we may and will assume that $\EE|X_0|^2<\infty$, where $K$ is a normalizing constant making $\mathbb{Q}$ a probability measure.
By Lemma \ref{lem2.1}, we see that  for all $t\ge 0$,
\[\sup\limits_{n\ge 1} \mathbb{E} [\sup\limits_{0\le s<t}|X^n_s|^2]<\infty.\]
Then
for every fixed $t\ge0$ the sequence of random variables $\{X^n_t\}$ is tight.

Let $\{\tau_n\}$ be a sequence of stopping times bounded above by $T\ge 0$. It is easy to calculate and to estimate by the boundedness of the coefficients that
\[\sup\limits_{n\ge1} \mathbb{E} \left[|X^n_{\tau_n+\de}-X^n_{\tau_n}|^2 \right]\le K_1\de\]
which tends to 0 as $\de\to 0$, where $K_1$ is a constant. The tightness of $\{X^n\}$ in $D([0,\infty),\mathbb{R})$ then follows from the criterion of Aldous \cite{AD} (see also Ethier and Kurtz \cite{EK}, pp. 137-138).
\end{proof}

Let $X$ be a limit of the sequence $\{X^n\}$. We  proceed to showing that $X$ is a solution to SDE (\ref{eq1}). Because of the discontinuity of the coefficients, we need the following result on the amount of time   $X$ spends at the discontinuity points.

Given $c\in \mathbb{R}$, the level set of process $X$  at level $c$ is defined as $\{t\geq 0:X_t=c\}$.

\begin{lemma}\label{level} Under conditions (2.a,b), for any level $c$ the level set has Lebesgue measure $0$ \, $\mathbb{P}$-a.s.
\end{lemma}
\begin{proof}
Without loss of generality, we consider the case of $c = 0$.
Denote $X^n_t = M^n_t - Z_t^n$ with
\[M_t^n := X_0 +\int_0^t \sigma_n(X_{s}^n)dB_s\]
and
\[Z^n_t:=-\int_0^t b_n(X_{s}^n)ds- \int_0^t \int_{U} g_{n}(X_{s-}^n, u) N(ds,du).\]

The plan of the proof is to show that $X_t=M_t-Z_t$, where $M_t$ is a martingale and $Z_t$ is of finite variation. We then prove that $M_t$ is of infinite variation over any set with positive Lebesgue measure, and hence, they cannot coincide on a such set.

 We denote by $\<M^n\>$ the quadratic variation of $M^n$,  and $|F|$ the Lebesgue measure of $F$. Similar to Lemma \ref{lem_tight}, we can prove the tightness of the process $(X^n,M^n,Z^n,\<M^n\>)$.  Let $(X,M,Z,A)$ be a limit  of $(X^n, M^n, Z^n, \<M^n\>)$ in distribution. Without loss of generality, we assume that $(X^n, M^n, Z^n, \<M^n\>)$ converges to $(X,M,Z,A)$ in distribution.
 By Skorohod's representation,
we may and will assume that $(X^n, M^n, Z^n, \<M^n\>)$ converges to $(X,M,Z,A)$ a.s.
Passing the relations to the limit, we see that the processes $M_t$ and $M^2_t-A_t$ are martingales, $M_t$ is continuous, and the equality
$X_t=M_t-Z_t$ holds.

For any function $f$ on $[0,T]$ and set $F\subset[0,T]$, we define the total variation of $f$ over $F$ as
\[V(f,F)=:\sup\sum^k_{i=1}|f(t_i)-f(t_{i-1})|,\]
where the supremum is taken over all partitions $0=t_0<t_1<\cdots<t_k=T$ with $t_i\in F$. Note that
\begin{eqnarray*}
\EE V(Z^n,[0,T])&=&\EE\int^T_0|b_n(X^n_s)|ds+\int^T_0\left|\int_U g_n(X^n_s,u)\mu(du)\right|ds\\
&\le&KT.
\end{eqnarray*}
By Fatou's lemma, we see that
\[\EE V(Z,F)\le\EE V(Z,[0,T])\le\liminf_{n\to\infty}\EE V(Z^n,[0,T])\le KT.\]

Let $F := \{t\in[0,T]: M_t = Z_t \}$. We proceed to prove that $|F|=0$ \,\, $\PP$-a.s. by contradiction.

Suppose that  $P(|F|>0)>0$. Let $\om\in\Omega$ be fixed such that  $|F(\omega)|>0$. We omit $\om$ in the rest of the proof.
Since $M_t$  is continuous, $\forall \ep>0$, $\exists \delta=\de(\ep)\in (0,\ep)$, for any $s, t\in [0,T]$, we have $|M_t-M_s|<\epsilon$, whenever $|t-s|<\delta$. Now we can find a finite number of points $0\leq a_1<b_1\leq a_2<b_2\leq\cdots\leq a_m<b_m$ such that $F\subset \cup_{i=1}^m [a_i,b_i]$, $|b_i-a_i|<\delta$ and $\sum^m_{i=1}|b_i-a_i|\le |F|+\ep$. 

Since $a_i$ and $b_i$ might be outside of $F$, we now modify them to $\bar{a}_i$ and $\bar{b}_i$ in $F$ which are close to $a_i$ and $b_i$, respectively. Choose
\[\bar{a}_i:=\inf \{[a_i,b_i]\cap F \}\mbox{ and }\bar{b}_i:=\sup\{ [a_i,b_i]\cap F\}.\]
Then,
\[(a_i,\bar{a}_i)\cap F=(\bar{b}_i,b_i)\cap F=\emptyset.\]
Let $J:=\{i:\bar{a}_i=\bar{b}_i\}$, then $\cup_{i\notin J}[\bar{a}_i, \bar{b}_i]\setminus F$ is countable.  For any $i\notin J$, if $\bar{a}_i\notin F$ or $\bar{b}_i\notin F$, we  choose  $a_i'\in F\cap (\bar{a}_i, \bar{a}_i+\frac{\epsilon}{2^i}\wedge\frac{\bar{b}_i-\bar{a}_i}{2})$ and $b_i'\in F\cap (\bar{b}_i-\frac{\epsilon}{2^i}\wedge \frac{\bar{b}_i-\bar{a}_i}{2}, \bar{b}_i)$  such that
\[0<\bar{b}_i-\bar{a}_i-\frac{\ep}{2^{i-1}}<b'_i-a'_i.\]

Next, we take $\ep\to 0$ (and hence, $\de\to 0$ and $m=m(\ep)\to\infty$). It is well-known that
\begin{align}\label{limit}
\lim_{\ep\to 0} \sum_{i=1}^{m(\ep)} \left(\(M_{b_{i}^{'}} - M_{a_{i}^{'}}\)^2 - \(A_{b_{i}^{'}} - A_{a_{i}^{'}}\) \right) = 0 \quad  \mbox{in probability.}
\end{align}
Then there is a subsequence $\ep_k\to 0$ such that
\begin{align}\label{eq1022a}
\lim_{k\to\infty} \sum_{i=1}^{m(\ep_k)} \left(\(M_{b_{i}^{'}} - M_{a_{i}^{'}}\)^2 - \(A_{b_{i}^{'}} - A_{a_{i}^{'}}\) \right) = 0 ,\ a.s.
\end{align}
Without loss of generality, we assume that (\ref{eq1022a}) holds for the $\om$ fixed above.
Thus, for any $\eta>0$, for $k$ large enough, we have
\begin{align*}
\sum_{i=1}^{m(\ep_k)} \(M_{b_{i}^{'}} - M_{a_{i}^{'}}\)^2  &\ge \sum_{i=1}^{m(\ep_k)}  \left(A_{b_{i}^{'}} - A_{a_{i}^{'}}\right) - \eta \\
&= \sum_{i=1}^{m(\ep_k)}  \lim_{n \to \infty}\left(\langle M^n \rangle_{b_{i}^{'}} - \langle M^n \rangle_{a_{i}^{'}}\right) - \eta \\
&\ge\sum_{i=1}^{m(\ep_k)}\sigma_0^2 \(b_{i}^{'} - a_{i}^{'}\) -\eta  \\
&\ge \sigma_0^2 \left(\sum_{i=1}^{m(\ep_k)} \(\bar{b}_{i} - \bar{a}_{i}\) - 2 \ep_k \right) - \eta \\
&\ge \sigma_0^2  |F| - 2 \sigma_0^2 \ep_k-\eta,
\end{align*}
where condition (2.b) is needed for the third inequality. On the other hand,
\begin{align*}
\sum_{i=1}^{m(\ep_k)} (M_{b_{i}^{'}} - M_{a_{i}^{'}})^2 \le \ep_k \sum_{i=1}^{m(\ep_k)} |M_{b_{i}^{'}} - M_{a_{i}^{'}}| = \ep_k \sum_{i=1}^{m(\ep_k)} |Z_{b_{i}^{'}} - Z_{a_{i}^{'}}|  \le \ep_k V(Z, F).
\end{align*}
Thus,
\[
\sigma_0^2  |F| - 2 \sigma_0^2 \ep_k-\eta \le \ep_k  V(Z, F).
\]
Letting $k\to\infty$, we get $\si^2_0|F|\le\eta$. We reach a contradiction by taking $\eta\to 0$.
\end{proof}

\begin{theorem}
Under conditions (A)  and (2.a,b) there exists a weak solution to SDE (\ref{eq1}).
\end{theorem}
\begin{proof}
As in the proof of Lemma \ref{level}, we continue to
  assume those processes are defined on the same probability space and $\{X_t^{n}: t \ge 0\}$ converges to $\{X_t: t \ge 0\}$  in $D([0, \infty), \mathbb{R})$ a.s.

Let \[\Lambda:=\{t\geq 0:\; X_t=X_{t-} \mbox{ is not a discontinuity point of functions $b$, $\si$, or $g$}\}.\]
By Lemma \ref{level} where we need conditions (2.a,b), we have $\text{Leb}(\Lambda^{c})
= 0$ a.s. Since $\xi_n \to 0$ in probability, for all $s \in
\Lambda$,
\[
\lim_{n \to \infty}b_n(X_{s}^n) = \lim_{n \to \infty} \mathbb{E}_{\xi_n} b(X_{s}^n + \xi_n) = b(X_{s}),
\]
where $\mathbb{E}_{\xi_n}$ is the expectation with respect to the random variable $\xi_n$.
Similar identities also hold for $\sigma_n$ and $g_n$.

 Let $\mathbb{B}$ be the space of bounded functions on  $\mathbb{R}$ equipped with the norm
$$\|f\|_0=\sup_x |f(x)|,\qquad\forall\ f\in\BB,$$ and
$\mathbb{B}_2:=\{ f \in C^2(\mathbb{R}) \cap \mathbb{B}:  f^{'}, f^{''} \in \mathbb{B}  \}$.
Then for any $f\in \mathbb{B}_2$,
\begin{eqnarray}\label{mp123}
M^{n,f}_t &=&f(X^n_t) - \int^t_0 \left(b_n(X^n_{s})f'(X^n_{s})+ \frac{1}{2} f''(X^n_{s})\sigma_n^2(X^n_{s})\right)ds\nonumber\\
&&- \int_0^t \int_{U} \left( f(X^n_{s}+g_{n}(X^n_{s},u)) - f(X^n_{s}) \right)\mu (du)ds
\end {eqnarray}
is a martingale.
Thus,  for  fixed $t>0$,
\begin{eqnarray}\label{mp123a}
\lim_{n \to \infty}M^{n,f}_t &=&\lim_{n \to \infty}f(X^n_t) - \lim_{n \to \infty}\int^t_0  \left(b_n(X^n_{s})f'(X^n_{s})+ \frac{1}{2} f''(X^n_{s})\sigma_n^2(X^n_{s})\right)ds\nonumber\\
&&- \lim_{n \to \infty}\int_0^t \int_{U} \left( f(X^n_{s}+g_{n}(X^n_{s},u)) - f(X^n_{s}) \right)\mu (du)ds\nonumber\\
&=&f(X_t) - \int^t_0 \left(b(X_{s})f'(X_{s})+ \frac{1}{2} f''(X_{s})\sigma(X_{s})^2\right)ds\nonumber\\
&&- \int_0^t \int_{U}\left(f(X_{s}+g(X_{s},u)) - f(X_{s})\right)\mu (du)ds\nonumber\\
&\equiv&M^{f}_t , \quad a.s.
\end {eqnarray}
and $M_t$ is a martingale.  Then process $X$ is a solution to the $A$-martingale problem. Therefore, it is a weak solution of (\ref{eq1}),  and the weak existence  follows.
\end{proof}

\section{Weak uniqueness}
\setcounter{equation}{0}
\renewcommand{\theequation}{\thesection.\arabic{equation}}

It is well-known that the weak uniqueness of the SDE (\ref{eq1}) is equivalent to the uniqueness of the solution to the $A$-martingale problem,  which is further equivalent to the uniqueness of  marginal distributions at any fixed time for any two solutions $X$ and $Y$ of the $A$-martingale problem, i.e. $X_t$ and $Y_t$ follow the same distribution for any $t\ge 0$ (see, for example, Theorem 4.2 on page 184 of \cite{EK}). Throughout this paper, we do not distinguish these three types of weak uniqueness.

In this section, we establish the weak uniqueness of the solution to SDE (\ref{eq1}).
To this end,  we impose the following conditions:

(3.a) There exists a constant $K \ge 0$ such that
$$ |b(x)| + \int_{U}  |g(x, u)|  \mu(du)    \le  K \sigma(x)^2, \quad \forall x \in \mathbb{R}.$$

(3.b)  There exists a constant $K \ge 0$ such that $0<|\sigma(x) | \le K, \quad  \forall x \in \mathbb{R}$.

The main tool of this section is Proposition \ref{prop1} below which extends a result of Kurtz and
Ocone \cite{KO}. Denote  by $\mathscr{P}(\mathbb{R})$
the set of all Borel probability measures on  $\mathbb{R}$.
 For any operator $A$, we denote its domain by $\cD(A)$ and its range by $\cR(A)$. Given $\mu \in \mathscr{P}(\mathbb{R})$ and a Borel measurable and $\mu$-integrable $f$ on $\mathbb{R}$, we write $\mu f:= \int_{\mathbb{R}} f(x) \mu (dx) $.

\begin{definition}
 We say that $M \subset
\mathbb{B}$ is separating (for $\mathscr{P}(\mathbb{R})$) if given $v,
\mu \in \mathscr{P}(\mathbb{R})$, $vf= \mu f$ for all $f \in M$
implies $v = \mu$.
\end{definition}

The following result is key to showing the weak uniqueness.

\begin{proposition}\label{prop1}
Suppose that there exists $\la_0\ge 0$ such that $\cR(\lambda - A)$ is separating for each $\lambda > \la_0$. If  $(\nu_t),\, (\mu_t)\subset\mathscr{P}(\mathbb{R})$ are weakly right continuous with $\nu_0=\mu_0$, and satisfy
\begin{equation}\label{eq0430a}
v_tf = v_0f + \int_0^t v_s Afds,
\quad \forall f \in \cD(A), t\geq 0,
 \end{equation}
 then $\nu_t = \mu_t$ for all $t \geq 0$.
\end{proposition}
\begin{proof}
The case of $\la_0=0$ was proved by Kurtz and
Ocone \cite{KO}. Taking the Laplace transforms on both sides of (\ref{eq0430a}), we have
\[
\lambda \int_0^{\infty} e^{-\lambda t}\nu_t f dt = \nu_0 f +\int_0^{\infty} e^{-\lambda t} \nu_t Af dt.
\]
Consequently,
\begin{eqnarray} \label{eq0123f}
\int_0^{\infty} e^{-\lambda t} \nu_t (\lambda f- Af) dt = \nu_0 f.
\end{eqnarray}
The same argument yields
\begin{eqnarray}\label{eq0123g}
\int_0^{\infty} e^{-\lambda t} \mu_t (\lambda f- Af) dt = \mu_0 f=\nu_0f.
\end{eqnarray}
Since $\cR(\lambda - A)$ is separating for all $\lambda \ge \lambda_0$, equations (\ref{eq0123f}) and  (\ref{eq0123g}) imply that as measures,
$$\int_0^{\infty} e^{-\lambda t} \nu_t dt= \int_0^{\infty} e^{-\lambda t} \mu_t dt, \quad \forall \lambda \ge \lambda_0. $$
Then for any $g\in\BB$ and $r:=\lambda - \lambda_0,$ $r\ge 0$,
\begin{eqnarray}
0&=&\int_0^{\infty} e^{-\lambda t} (\nu_t - \mu_t)g dt\nonumber\\
 &=& \int_0^{\infty} e^{-(\lambda- \lambda_0 )t} e^{\lambda_0 t}(\nu_t - \mu_t)g dt  \nonumber\\
&=&\int_0^{\infty} e^{-r t} e^{\lambda_0 t}(\nu_t - \mu_t) g dt.
\end{eqnarray}
Since $(\nu_t)$ and $(\mu_t)$ are weakly right continuous, the uniqueness of the Laplace transform implies $ e^{\lambda_0 t}(\nu_t - \mu_t)g =0$, and hence,
$$\nu_t = \mu_t, \quad t \ge 0.$$
\end{proof}

Now we are ready to prove the weak uniqueness. We first make a time change so that the diffusion part of the process becomes a Brownian motion. We then establish the uniqueness of the time-changed equation by verifying the conditions of Proposition \ref{prop1}. At the end, we convert the uniqueness result to that for the original equation.

From SDE (\ref{eq1}), for any $f\in C^2_b(\mathbb{R})$, $M^f_t$ defined in (\ref{mp123a}) is a martingale. Let
\[\tau_t = \int_0^t \sigma(X_s)^2  ds\mbox{ and }\widetilde{X}_t = X_{\tau_t^{-1}},\]
where $\tau_t^{-1}$ is the inverse of $\tau$. Note that $\tau_t\rightarrow \infty $ as $t\rightarrow\infty$ since $\si$ is bounded below away from 0 uniformly.

Let
\[W_t=\int_0^{\tau_t^{-1}}\sigma(X_{s})dB_s.\]
Then $W_t$ is a continuous martingale with $\<W\>_t=\int^{\tau^{-1}_t}_0\si(X_s)^2ds=t$. By \levy's characterization theorem, $W_t$ is a Brownian motion.

Changing the variable $s = \tau_r^{-1}$, it is easy to see that
$$ds = \frac{dr}{\sigma(X_s)^2} = \frac{dr}{\sigma^2(\widetilde{X}_r)}.$$
By (\ref{eq1}) we have
\begin{equation}\label{equ_change}
\widetilde{X}_t = X_0+W_t+\int^t_0(\sigma^{-2}b)(\widetilde{X}_{s})ds
 +\sum_{s\in D, s\le \tau^{-1}_t}g(X_{s-},p_s),
\end{equation}
where $D$ is the set of jumping times.
To express  equation (\ref{equ_change}) as  an equation for the time changed process $\tilde{X}$, we define a new random measure $N_1$ by
\[N_1((0,t]\times A)=\sum_{s\in D, s\le \tau^{-1}_t}1_{p_s\in A},\qquad  \forall A\in{\cB(U)}.\] Then, $N_1$ is a random measure with compensator
\[\frac1{\si(\tilde{X}_r)^2}dr\mu(du)=\int_{\RR_+}1_{a\le \si(\tilde{X}_r)^{-2}}dadr\mu(du).\]
 By Theorem 7.4 of (\cite{IW}, p93), we can express its dependence on the underlying process $\tilde{X}$ more explicitly by
\[N_1(dr,du)=\int^{\si(\tilde{X}_r)^{-2}}_0N_2(dr,du,da),\]
where $N_2$ is a Poisson random measure on $\RR_+\times U\times\RR_+$ with intensity measure  $dr\mu(du)da$. Therefore,  equation (\ref{equ_change}) can be rewritten as
\begin{align} \label{timechange}
\widetilde{X}_t = X_0+W_t+\int^t_0(\sigma^{-2}b)(\widetilde{X}_{s})ds
 + \int_0^t \int_{U} \int^{\si(\tilde{X}_s)^{-2}}_0    g(\tilde{X}_{s-}, u) N_2 \(ds, du,da\).
\end{align}

To  consider discontinuous test functions, we extend  the semigroup of the Brownian motion to the space $\mathbb{B}$ as follows
\[T_tf(x):=\int_{\mathbb{R}} p_t(x-y)f(y)dy,\qquad \forall\ f\in\mathbb{B}, x\in\mathbb{R},\]
where
\[
p_t (x-y) := \frac{1}{\sqrt{2 \pi t}}e^{ -\frac{(x-y)^2}{2t}}.
\]
Let $A_0$ be the infinitesimal generator of the semigroup $\{T_t, t \ge 0\}$ on $\mathbb{B}$ with
 $A_0f:=\frac12 f''$ for $f\in\cD(A_0):=\{f:\ f,\ f',\ f''\in\mathbb{B}\}$.

Let $\cD(\tilde{A}):=\cD(A_0)$ and $\tilde{A}:=A_0+B+C $,  where for any $f\in \cD(\tilde{A})$,
\begin{align*}
B f(x)  := \frac{1}{\sigma(x)^2} \int_{U}\(f\(x+g(x, u)\)-f(x)\)\mu(du),
\end{align*}
and
\[Cf(x) := \frac{b(x)}{\sigma(x)^2} f'(x).\]
Applying It\^o's formula to (\ref{timechange}) and taking an expectation, we have
\begin{equation}\label{time change}
\mathbb{E}f(\widetilde{X}_t)=\mathbb{E}
f(\widetilde{X}_0)+\int^t_0\mathbb{E}\tilde{A}f(\widetilde{X}_s)ds.
\end{equation}

\begin{theorem}
 Under conditions (A) and (3.a,b), the weak uniqueness holds for the time changed SDE
(\ref{timechange}), and consequently, it also  holds for the original SDE (\ref{eq1}).
\end{theorem}
\begin{proof} 
We are going to show that $\cR(\lambda-\tilde{A})$ is separating. For any $\lambda > 0$, define $R_{\lambda} f := \int_0^{\infty} e^{-\lambda t} T_t f dt $.
Given $g\in \mathbb{B}$, we want to find $f\in \cD(\tilde{A})$ such that $(\lambda-\tilde{A})f=g$, or equivalently,
\begin{equation}\label{eq0123w}
(\lambda-A_0)f=g+(B+C)f.
\end{equation}
To this end, we apply $R_\la$ to both sides of (\ref{eq0123w}) and consider the following equation
\begin{equation}\label{eq0123c}
f=R_{\lambda}(g+Bf+Cf )=: \Gamma_g(f).
\end{equation}

Let $\BB_1:=\{f\in C^1(\mathbb{R})\cap\mathbb{B}:\; f' \in \mathbb{B}\}$ with norm
\[\|f\|_1=\sup_x|f(x)|+\sup_x|f'(x)|.\]
To solve (\ref{eq0123c}), we first prove that  for $\lambda$ large enough, $\Ga_g$ is a contraction mapping from $\BB_1$ to $\BB_1$,
i.e. there exists a constant $0<L<1$ such that
  $$\| \Ga_g (f_1 - f_2) \|_1 \le L \| f_1 - f_2 \|_1 , \quad\forall f_1, f_2 \in \BB_1.$$

Let $\hat{f} :=  f_1 - f_2.$ It is clear that both $B$ and $C$ are bounded linear operators from $\BB_1$ to
$\mathbb{B}$ since $\forall x\in\mathbb{R}$,  by condition (3.a),
\begin{align*}
| B \hat{f}(x) | =& \frac{1}{\sigma(x)^2} \left| \int_{U}\left(\hat{f}(x+g(x, u))-\hat{f}(x) \right)\mu(du) \right| \\
\le &  \frac{\|\hat{f}\|_1}{\sigma(x)^2} \left| \int_{U} g(x, u) \mu(du) \right|
\le  K\| \hat{f} \|_1 ,
\end{align*}
and
\[ \left|C\hat{f}(x)\right| = \frac{| b(x) |}{\sigma(x)^2}| \hat{f}'(x)|  \le K \| \hat{f} \|_1.\]
On the other hand, for $\hat{f}\in \mathbb{B}$, we have
\[|T_t \hat{f}(x)|\le \|\hat{f}\|_0, \,\,\forall x\in\mathbb{R}\]
and
\begin{eqnarray}\label{eq0124a}
|\partial_xT_t\hat{f}(x)|&=&\left|\int_{\mathbb{R}} \partial_x p_t(x-z) \hat{f} (z)dz\right|
\nonumber\\
&=& \left| \int_{\mathbb{R}} - \frac{1}{\sqrt{2\pi t}} e^{- \frac{{(x-z)}^2}{2t}} \frac{(x-z)}{t} \hat{f} (z) dz \right| \nonumber\\
&\le&\frac{K}{\sqrt{t}}\int_{\mathbb{R}} p_{2t}(x-z)dz\|\hat{f}\|_0=\frac{K}{\sqrt{t}}\|\hat{f}\|_0. \nonumber
\end{eqnarray}
Hence, $T_t$ is a bounded linear operator from $\mathbb{B}$ to $\BB_1$. In addition,
\begin{align*}
\| T_t \|_{L(\mathbb{B}, \BB_1)} &= \sup_{\| \hat{f} \|_{\mathbb{B} \le 1}}  \| T_t \hat{f} \|_{\BB_1} \\
&= \sup_{\| \hat{f} \|_{\mathbb{B} \le 1}} \(\sup_x |T_t \hat{f} (x)|+\sup_x|  \partial_xT_t\hat{f}(x)|  \)  \\
&\le  K (1 + \frac{K}{\sqrt{t}}).
\end{align*}
It follows that  for $\hat{f} \in \BB_1$,
\begin{eqnarray}
\| \Ga_g (\hat{f}) \|_1 &\le& \int_0^{\infty} e^{-\lambda t } \| T_t\|_{L(\mathbb{B}, \BB_1)} \|(B+C)\|_{L(\BB_1, \mathbb{B})} \|\hat{f}\|_1dt \nonumber \\
&\le & K \int^\infty_0 e^{-\lambda t}(1+ Kt^{-1/2})dt \|\hat{f}\|_1. \nonumber
\end{eqnarray}
Taking $\lambda$ large enough, there exists $0<L<1$ such
that $\Ga_g$ is a bounded linear operator on
$\BB_1$ whose norm is bounded by $L$. Therefore, $\Gamma_g$ is
a contraction mapping on $\BB_1$.  Choose
$\lambda_0$ large enough such that $0<L<1$. Then,  for  $\lambda \ge
\lambda_0$, there exists a unique $f\in\BB_1$ such that
(\ref{eq0123c}) is satisfied.

 To prove that $f$ is the solution of (\ref{eq0123w}), we denote $h := g + Bf +Cf \in \BB$, and consider the following ordinary differential equation (ODE in short):
\begin{equation}\label{eq0123d}
\lambda l(x) - \frac{1}{2} l''(x) = h(x)
\end{equation}
Solving the above ODE, we obtain
\[ l(x) = e^{\sqrt{2\lambda}x}\int_x^{\infty}   \frac{h(y)}{\sqrt{2\lambda}}e^{-\sqrt{2\lambda}y}dy +e^{-\sqrt{2\lambda}x}  \int_{-\infty}^x  \frac{h(y)}{\sqrt{2\lambda}}e^{\sqrt{2\lambda}y}dy.
\]
It is straightforward to verify that $l\in\cD(\tilde{A})$. Next, we want to show that  $l$ is also the solution to (\ref{eq0123c}).

Applying $R_{\lambda}$ to both sides of (\ref{eq0123d}), we have
\[R_{\lambda}(\lambda l - \frac{1}{2} l'' ) = R_{\lambda}h.  \]

Since $R_{\lambda} h = f$, and
\begin{eqnarray}
R_{\lambda}(\lambda l - \frac{1}{2} l'') &=&\lambda \int_0^{\infty} e^{-\lambda t} T_t l(x) dt - \frac{1}{2} \int_0^{\infty} e^{-\lambda t} T_t l''(x) dt   \nonumber  \\
&=& \lambda \int_0^{\infty} e^{-\lambda t} T_t l(x) dt - \int_0^{\infty} e^{-\lambda t} \frac{d}{dt} \left(T_t l(x)\right) dt  \nonumber  \\
&=& l,   \nonumber
\end{eqnarray}
we obtain that $f =l\in\cD(\tilde{A})$ is a solution to (\ref{eq0123d}) and thus to (\ref{eq0123w}).  This implies that $\cR(\lambda-\tilde{A})=\mathbb{B}$, and hence, it is separating for all $\lambda \ge \lambda_0$.   The weak uniqueness for the time changed SDE (\ref{timechange}) then follows from Proposition \ref{prop1}.

Finally, we proceed to prove the weak uniqueness of the original SDE (\ref{eq1}). Suppose that $X$ and $Y$ are any two solutions of  SDE  (\ref{eq1}).
Let
  \[\tau_t := \int_0^t \sigma(X_s)^2 ds\mbox{ and }\lambda_t := \int_0^t \sigma(Y_s)^2 ds.\]
   The  time changed processes $\widetilde{X}_t$ and $\widetilde{Y}_t$ satisfy  SDE (\ref{timechange}). Hence, $\widetilde{X}$ and  $\widetilde{Y}$ have the same law.
Notice that for all $t\geq 0$,
$$\tau_t^{-1} = \int_0^t \frac{1}{\sigma(\widetilde{X}_s)^2}ds
\quad\text{and}\quad \lambda_t^{-1} = \int_0^t \frac{1}{\sigma(\widetilde{Y}_s)^2}ds$$ 
are the same function of $\widetilde{X}$ and $\widetilde{Y}$, respectively. Then $\tau_t$ and  $\lambda_t$ are also the same function of  $\widetilde{X}$ and $\widetilde{Y}$, respectively. Since $X_t = \widetilde{X}_{\tau_t}$ and $Y_t = \widetilde{Y}_{\lambda_t}$ for all $t\geq 0$,
the weak uniqueness of the original SDE (\ref{eq1}) thus follows.\\
\end{proof}

\section{Pathwise uniqueness }
\setcounter{equation}{0}
\renewcommand{\theequation}{\thesection.\arabic{equation}}

To show the pathwise uniqueness, we need the following conditions:
\begin{itemize}
\item[(4.a)] For any fixed $u$, $g(x,u)+x$ is non-decreasing in $x$;
\item[(4.b)] There exist constants $\sigma_0, K \ge 0$, such that $0<\sigma_0 \le
| \sigma(x) | \le K$ for all $x $.
\end{itemize}
Note that (4.a) is a classical condition for comparison theorems for jumping SDEs.

The following Tanaka's formula can be found in
Theorem 68 of Protter \cite{PE}. We state it here for the convenience of the reader.

\begin{theorem}
(Tanaka's Formula) Let $X$ be a semimartingale and let $L^a$ be its local time at a. Then
\begin{align*}
(X_t - a)^{+} - (X_0 - a)^{+} =& \int_{0+}^t \mathbf{1}_{\{X_{s-} >a\}} dX_s + \sum_{0<s\le t}\mathbf{1}_{\{X_{s-} >a\}} (X_s - a)^{-}  \nonumber \\
&+\sum_{0<s\le t}\mathbf{1}_{\{X_{s-} \le a\}} (X_s - a)^{+} + \frac{1}{2} L_t^a
\end{align*}
and
\begin{align*}
(X_t - a)^{-} - (X_0 - a)^{-} =& -\int_{0+}^t \mathbf{1}_{\{X_{s-} \le a\}} dX_s + \sum_{0<s\le t}\mathbf{1}_{\{X_{s-} >a\}} (X_s - a)^{-}  \nonumber \\
&+\sum_{0<s\le t}\mathbf{1}_{\{X_{s-} \le a\}} (X_s - a)^{+} + \frac{1}{2} L_t^a.
\end{align*}
\end{theorem}

\begin{proposition}
Under Condition (4.a), if $X^1$ and $X^2$ are two solutions of (\ref{eq1}) such that $X_0^1 = X_0^2$ \,\, $\mathbb{P}$-a.s., then  $X^1 \vee X^2$ is a solution of (\ref{eq1}) if and only if  $L^0(X^1 - X^2)\equiv 0$ \, $\mathbb{P}$-a.s.
\end{proposition}
\begin{proof}
Applying Tanaka's formula to $(X^2_t-X^1_t)^+$, we obtain
\begin{align}  \label{eq1234}
X_t^1 \vee X_t^2 =& X_t^1 + (X_t^2 - X_t^1)^{+}  \nonumber \\
=&  X_t^1 +\int_{0+}^t  \mathbf{1}_{\{X_{s-}^2 >X_{s-}^1\}} d(X^2 - X^1)_s + \sum\limits_{0 < s \le t}  \mathbf{1}_{\{X_{s-}^2 >X_{s-}^1\}} (X_s^2 - X_s^1) ^{-}  \nonumber\\
&+  \sum\limits_{0 < s \le t}  \mathbf{1}_{\{X_{s-}^2 \le X_{s-}^1\}} (X_s^2 - X_s^1) ^{+} +   \frac{1}{2}L_t^0(X^2 - X^1).
\end{align}
For $s\in D$, the collection of jumping times of $N$, we have
\[X_s^2 - X_s^1 =    X_{s-}^2 - X_{s-}^1+ g\left( X_{s-}^2, p_s\right) - g\left (X_{s-}^1,  p_s\right). \]
By the non-decreasing property of $x+ g(x,u)$ in $x$  from condition (4.a), we see that $X_{s-}^2 >X_{s-}^1$ implies $X^2_s-X^1_s\ge 0$, and hence,
\[
 \sum\limits_{0 < s \le t}  \mathbf{1}_{\{X_{s-}^2 >X_{s-}^1\}} (X_s^2 - X_s^1) ^{-}  =0.
\]
Similarly, $$\sum\limits_{0 < s \le t}  \mathbf{1}_{\{X_{s-}^2 \le X_{s-}^1\}} (X_s^2 - X_s^1) ^{+} = 0. $$
Substituting $X^i_t$, $i=1,\ 2$ on the RHS of (\ref{eq1234}) by their expressions given by SDE (\ref{eq1}), we arrive at
\begin{align*}
X_t^1 \vee X_t^2
=& X_0^1 \vee X_0^2 + \int_0^t \sigma(X^1_{s-} \vee X^2_{s-} ) dB_s + \int_0^t b(X^1_{s-} \vee X^2_s ) dB_s \\
&+ \int_{0+}^t \int_{U} g(X^1_{s-}  \vee  X^2_{s-}, u)  N(ds,du) \\
&+ \frac{1}{2}L_t^0(X^2 - X^1) .
\end{align*}
The conclusion of the proposition follows directly from the equation above.
\end{proof}

\begin{proposition} \label{prop1a}
Under Condition (4.a), if the weak uniqueness  holds for SDE (\ref{eq1}) and $L^0(X^1-X^2)\equiv 0$ for any pair  of solutions $(X^1, X^2)$ to (\ref{eq1})  with $X_0^1 = X_0^2$ a.s., then the pathwise uniqueness holds for (\ref{eq1}).
\end{proposition}
\begin{proof}
If $X^1$ and $X^2$ are any two solutions  and $L^0(X^1-X^2)=0$, then $X^1 \vee X^2$ is also a solution. Since the weak uniqueness holds,  we have that processes
$X^1, X^2$ and $X^1 \vee X^2$ have the same law.
 For any $t\geq 0$, $X^1_t \vee X^2_t - X^1_t$ is a non-negative random variable. We then have $X^1_t \vee X^2_t = X^1_t$  a.s. Similarly, we have $X^1_t \vee X^2_t = X^2_t$ a.s. Therefore, $X^2_t = X^1_t$ a.s.  Then $X^1=X^2$ a.s. by the right continuity of $X^1$ and $X^2$.
\end{proof}

The next lemma is crucial to verifying   condition (LT).

\begin{lemma}\label{lemma1}
Under conditions (2.a) and (4.a), if $X^1$ and $X^2$ are two
solutions of (\ref{eq1}), then $\mathbb{E} \left[L_t^a (X^1 - X^2 )\right]
\to \mathbb{E} \left[ L_t^0 (X^1 - X^2 )\right] $   as $ a \to 0$.
\end{lemma}

\begin{proof}
Without loss of generality, we take $a\geq 0$. Recall that $D$ is the set of jumping times.  By Tanaka's formula,
\begin{align}\label{eq1013a}
\frac{1}{2}L_t^a(X^1 - X^2) =&\( (X_t^1-X_t^2 - a)^{-} - (X_0^1-X_0^2 - a)^{-}\)\nonumber\\
& + \int_{0+}^t \mathbf{1}_{\{X_{s-}^1-X_{s-}^2 \le a\}}d(X_s^1-X_s^2) \nonumber\\
&- \sum\limits_{s\in(0,t]\cap D}\mathbf{1}_{\{X_{s-}^1-X_{s-}^2 > a\}}(X_s^1 - X_s^2 -a)^-\nonumber \\
&- \sum\limits_{s\in(0,t]\cap D}\mathbf{1}_{\{X^1_{s-}-X^2_{s-}\le a\}}(X_s^1 - X_s^2 -a)^+ \nonumber \\
=: & I_1^a +  I_2^a -  \sum\limits_{s\in(0,t]\cap D}\(I_3^a (s)+ I_4^a(s)\).
\end{align}

Since $|I_1^a - I_1^0 |\le 2a$,  then
$ \mathbb{E}(I_1^a - I_1^0) \to 0, $ as $  a \to 0  $. As $a\to 0$, it follows from the dominated convergence theorem (DCT) that
\begin{eqnarray*}
\EE\(I_2^a - I_2^0\)&=& \EE \int_0^t  \left(\mathbf{1}_{\{X_{s}^1-X_{s}^2 \le a\}} - \mathbf{1}_{\{X_{s}^1-X_{s}^2 \le 0\}}\right) \left(b(X_{s}^1)-b(X_{s}^2)\right) ds\\
&&+\EE \int_{0}^t   \int_{U} \left(\mathbf{1}_{\{X_{s}^1-X_{s}^2 \le a\}} - \mathbf{1}_{\{X_{s}^1-X_{s}^2 \le 0\}}\right)  \left(g(X_{s}^1, u)  - g(X_{s}^2, u)  \right) \mu(du)ds\\
&\to&0.
\end{eqnarray*}
Since $x+g(x,u)$ is non-decreasing in $x$, for $s\in D$, we have $\lim_{a\to 0+}I^a_4(s)=0=I^0_4(s)$ and
\begin{eqnarray}\label{eq333}
0\le I_4^{a}(s)& = &\mathbf{1}_{\{X^1_{s-}  \le X^2_{s-}\}} \left [ X_{s-}^1 - X_{s-}^2 -a+  g \left(X_{s-}^1, p_s\right) - g\left(X_{s-}^2, p_s\right)   \right]^+  \nonumber\\
&&+ \mathbf{1}_{\{X^2_{s-} < X^1_{s-} \le a + X_{s-}^2\}} \left [ X_{s-}^1 - X_{s-}^2 -a+ g \left(X_{s-}^1, p_s\right) - g \left(X_{s-}^2, p_s\right)  \right]^+  \nonumber \\
&\le&| g(X^1_{s-}, p_s)| +| g(X^2_{s-}, p_s)|=: h(s-).
\end{eqnarray}
Further, it is easy to see that
\[\EE\sum\limits_{s\in(0,t]\cap D}h(s-)<\infty.\]
By DCT again, we get
\[\lim_{a\to 0+}\EE \sum\limits_{s\in(0,t]\cap D}I^a_4(s)\to 0.\]
The same limit holds for $\EE \sum\limits_{s\in(0,t]\cap D}I^a_3(s)$.

The conclusion of the lemma follows by applying the estimates above to (\ref{eq1013a}).
\end{proof}

Let $\rho:\ \RR_+\to\RR_+$ be a Borel measurable function  such that $\int_{0+} da/\rho(a) = \infty$.

\begin{lemma} \label{lemma2}
Let $X$ be a semimartingale. For $\varepsilon >0$ and $t>0$ define
\[
 A_t^\varepsilon := \int_0^t \mathbf{1}_{\{0<X_s\le \varepsilon\}} \rho (X_s)^{-1} d [ X,X ]_s^c.
\]
If $ \mathbb{E} A_t^\varepsilon < \infty$ and $ \lim \limits_{a \to 0+} \mathbb{E} [L_t^a(X) ]  = \mathbb{E} [ L_t^0(X) ]$ for some $\varepsilon >0$ and all $t>0$,
then $L^0(X) = 0$ \, $\mathbb{P}$-a.s.
\end{lemma}
\begin{proof}
Fix $t>0$. By the occupation time formula (see Corollary 1 on Page 216 of \cite{PE}), for any positive Borel function $f$, we have
\[
\int_{-\infty}^{\infty}  L_t^a f(a) da = \int_0^t f(X_{s-}) d[ X, X ]_s^c.
\]
Then
\[
A_t^\varepsilon = \int_0^{\varepsilon} \rho(a)^{-1} L_t^a(X)da.
\]
If $L_t^0(X)$ does not vanish a.s., then $\mathbb{E} L_t^0(X)$ is positive. Since $\mathbb{E} L_t^a(X)$ converges to $\mathbb{E} L_t^0(X)$ when $a$ decreases to 0,  and $\int_{0+}\frac{da}{\rho(a)}=\infty$, we have $\mathbb{E} A_t^{\varepsilon} = \infty$, which contradicts from the assumption.
\end{proof}
\begin{theorem}
Suppose that Condition (A) holds. The pathwise uniqueness holds for (\ref{eq1}) in each of the following two cases:
\begin{enumerate}[(1)]
\item Conditions (3.a,b) and (4.a) hold, and $|\sigma(x)-\sigma(y)|^2 \le \rho(|x-y|)$.

\item Conditions (3.a) and (4.a,b) hold, and $| \sigma(x) - \sigma (y) |^2 \le |f(x) - f(y)|$  for a non-decreasing and bounded function $f$.
\end{enumerate}
\end{theorem}
\begin{proof}
We adopt arguments similar to those in Theorem 3.5 of \cite{DM}.

\noindent
{\it Proof of (1).} Under conditions (3.a,b), the weak uniqueness holds. Let $X^1$ and $ X^2$ be two solutions to (1.1) with respect to the same Brownian Motion. Then
\begin{align*}
  X_t^1 - X_t^2  & =  X_0^1-X_0^2 + \int_0^t \left(b(X_{s-}^1)-b(X_{s-}^2) \right) ds \\
&+ \int_0^t  \left(\sigma(X_{s-}^1)-\sigma(X_{s-}^2)\right) dB_s  \\
&+\int_0^t \int_{U} \left(g(X_{s-}^1, u)  - g(X_{s-}^2, u)  \right) N(ds,du).
\end{align*}
Observe that
\begin{eqnarray}
&& \mathbb{E} \left[\int_0^t \rho(X_s^1 - X_s^2)^{-1} \mathbf{1}_{\{X_s^1>X_s^2\}} d[ X^1 - X^2, X^1 - X^2 ]^c_s \right] \nonumber  \\
&&=\mathbb{E} \left[ \int_0^t \rho(X_s^1 - X_s^2)^{-1} \left(\sigma(X_s^1)-\sigma(X_s^2)\right)^2 \mathbf{1}_{\{X_s^1>X_s^2\}} ds\right] \le t. \nonumber
\end{eqnarray}
By Lemmas \ref{lemma1} and \ref{lemma2}, we have $L^0(X^1 - X^2) = 0$. Hence, by Proposition \ref{prop1a}, the pathwise uniqueness holds.

\noindent
{\it Proof of (2).}  Under conditions (3.a) and (4.b), the weak uniqueness holds. To
prove the pathwise uniqueness, we need to prove that $L^0 (X^2 -
X^1) = 0$ a.s. For any $t>0$ we consider the $A^\varepsilon_t$ in Lemma 4.2 with $\varepsilon = \infty$, $\rho(x)
= x$ and $X_t = X_t^1 - X_t^2$. Choose a $\delta >0$, and consider
\begin{eqnarray}
&\mathbb{E}& \left[ \int_0^t (X_s^1 - X_s^2)^{-1} \mathbf{1}_{\{X_s^1- X_s^2>\delta\}}  d [X^1 - X^2, X^1 - X^2]^c_s \right]  \nonumber \\
&\le& \mathbb{E} \left[ \int_0^t \left (f(X_s^1) - f(X_s^2)\right ) (X_s^1-X_s^2)^{-1}  \mathbf{1}_{\{X_s^1- X_s^2>\delta\}} ds\right] =: K(f)_t.  \nonumber
\end{eqnarray}
Now we construct a sequence of smooth functions
\[
f_n(x) := \mathbb{E} f(x + \xi_n), \quad \xi_n \sim N(0, \frac{1}{n}), \,\, n=1,2,\ldots.
\]
to approximate $f$. For each $n$, $f_n$ is  bounded, increasing and differentiable. Denote
$$ D_f := \{x\in\mathbb{R} : f(x) \text{\, is discontinuous at \,} x\},$$
which  is a countable set. Then
$\lim \limits_{n \to \infty} f_n(x) = f(x), \forall x \notin D_f$. By the occupation times formula
\[
0 = \int_{-\infty}^{+\infty} L_t^a \mathbf{1}_{\{a \in D_f\}} da = \int_0^t\mathbf{1}_{\{X_{s-} \in D_f\}} d[X,X]_s^c
\]
so that
\[
0=\int_0^t\mathbf{1}_{\{X_{s-} \in D_f\}} \sigma^2(X_{s-})ds \ge \sigma_0^2 \int_0^t\mathbf{1}_{\{X_{s-} \in D_f\}}ds = \sigma_0^2 \text{Leb} \{s: X_{s-} \in D_f\}.
\]
Then $\text{Leb} \{s: X_{s-} \in D_f\} = 0$ a.s., and $\text{Leb} \{s: X_{s-} \in D_f \text{\, or \,} X_s \ne X_{s-}\} = 0$ a.s. Consequently,
\[
\lim_{n \to \infty} \left(f_n(X_s^1) - f_n(X_s^2) \right) = f(X_s^1) - f(X_s^2),\qquad a.s.,
\]
for almost all $s \le t$. It follows that $K(f)_t = \lim \limits_{n \to \infty} K(f_n)_t$.

For any $v \in [0,1]$, put $Z^v_t := X_t^2 + v(X_t^1-X_t^2)$. Then
\begin{align}
Z_t^v = Z_0^v +  \int_0^t b^v_{s-} ds+ \int_0^t \sigma^v_{s-}  dB_s
 +  \int_0^t \int_{U} g^v_{s-} \left( u \right) N(ds, du),
\end{align}
where
\[
b^v_{s-}:= vb(X_{s-}^1) + (1 - v) b (X_{s-}^2),
\]
\[
\sigma^v_{s-} := v\sigma (X_{s-}^1) + (1 - v) \sigma (X_{s-}^2),
\]
\[
g^v_{s-}(u):= v g (X_{s-}^1,u) + (1 - v) g (X_{s-}^2,u).
\]

Since  $[Z^v, Z^v]^c_t=\int_0^t (\sigma^v_{s-})^2ds $, it follows that

\begin{eqnarray} \label{eqK}
K(f_n)_t
&=& \mathbb{E} \left[ \int_0^t \left( \int_0^1 f'_n(Z_s^v) dv \right) \mathbf{1}_{(X_s^1- X_s^2>\delta)} ds \right]   \nonumber  \\
&=& \int_0^1 \mathbb{E} \left[ \int_0^t f'_n(Z_s^v) ds \right] dv  \nonumber  \\
&=& \int_0^1 \mathbb{E} \left[ \int_0^t f'_n(Z_s^v)(\sigma^v_{s-})^{-2} d[Z^v, Z^v]^c_s  \right] dv  \nonumber  \\
&\le& \frac{1}{\sigma_0^2}\int_0^1  \mathbb{E} \left[ \int_0^t f'_n(Z_s^v) d[Z^v, Z^v]^c_s  \right] dv  \nonumber  \\
&\le& \frac{1}{\sigma_0^2} \int_0^1 \mathbb{E} \left[ \int_{\mathbb{R}} f'_n(a) L_t^a (Z^v) da \right]dv.
\end{eqnarray}

Moreover, $|\sigma^v| \le K$ and $|b^v| \le K$. By  Tanaka's formula, we have
\begin{align}\label{eqjumps}
(Z_t^v - a)^{+} =& (Z_0^v - a)^{+} +\int_0^t \mathbf{1}_{\{Z_{s-}^v > a\}} \left (  b^v_{s-}ds +\sigma^v_{s-} dB_s + \int_{U} g^v_{s-}( u) N(ds,du)  \right) \nonumber \\
& + \sum\limits_{0<s\le t} \mathbf{1}_{\{Z_{s-}^v > a\}} (Z_s^v - a)^{-}
+  \sum\limits_{0<s\le t} \mathbf{1}_{\{Z_{s-}^v  \le a\}} (Z_t^v - a)^{+} +\frac{1}{2} L_t^a(Z^v ).
\end{align}
Therefore,
\begin{align}
\frac{1}{2} \mathbb{E} L_t^a(Z^v) \le \mathbb{E}(Z_t^v - a)^{+} + \mathbb{E}\int_0^t |b_{s-}^v| ds + \mathbb{E} \int_0^t  \int_{U} |g^v_{s-}(u)| N(ds, du) <\infty,
\end{align}
and hence,
\[
\sup \limits_{a, v} \mathbb{E} \left[ L_t^a(Z^v)\right] = C < \infty.
\]
It follows from (\ref{eqK}) that
\[
K(f_n)_t \le \sigma_0^{-2} C \sup_n \| f_n\|.
\]
Hence, $K(f)_t$ is bounded by a constant which does not depend on $\delta$. Taking $\delta\to 0$, we see that  $K(f)_t$ is  bounded. By Lemma \ref{lemma1} and \ref{lemma2}, we have $L^0(X^1 - X^2) = 0$ a.s. By Propositions 4.1 and 4.2, the pathwise uniqueness holds for (\ref{eq1}).
\end{proof}

	In summary,  we have proved the pathwise uniqueness for SDE (\ref{eq1}) with possibly discontinuous coefficients under conditions that listed at the beginning of the previous sections. However, due to the limitation of our method, the jump part of the solution needs to be of bounded variation. More precisely, we have used the continuity (in spatial variable) of the local time which holds when the jumps part is of finite variation only. We leave the case of  jump part with unbounded variation  as a challenging {\em open problem}.

\section{An application}
\setcounter{equation}{0}
\renewcommand{\theequation}{\thesection.\arabic{equation}}

A modern approach in ruin theory is to use a spectrally negative {\levy} process to describe the surplus of an insurance company/portfolio. In actuarial mathematics literature, these {\levy} processes with negative jumps are also called {\levy} insurance risk processes.

The following equation specifies the so-called refracted {\levy} process.
\begin{eqnarray}\label{eq23}
dU_t = - \delta \mathbf{1}_{\{U_t >b\}}dt + dX_t
\end{eqnarray}
where $X = \{X_t: t \ge 0\}$ is a spectrally negative {\levy} process with law $\mathbb{P}$ and $b, \delta \in \mathbb{R}$ such that the resulting process $U$ may visit the half line $(b, \infty)$ with positive probability.

Note that the equation (\ref{eq23}) is motivated by an application in actuarial mathematics. In fact, the surplus process  $X_t$  without dividend payments is given by
\[dX_t=\mu dt+\si dB_t+dJ_t\]
where $\mu$ is the average premium  rate, $B$ is a Brownian motion and $J_t$ is an independent  pure jump spectrally negative {\levy} process; the second term comes from uncertainty of premium collection or other random factors, i.e., the insured will pay the premium with certain probability, and a scaling limit leads to this term; $J_t$ is the accumulated claims up to time $t$. $\de$ is the rate of dividend, i.e., the insurance company will pay dividends when the surplus is higher than a certain level. To summarize, the equation (\ref{eq23}) can be rewritten as
\[dU_t=\((\mu-\de)\mathbf{1}_{U_t > b}+\mu \mathbf{1}_{U_t \le b}\)dt+\si dB_t+dJ_t.\]

 Kyprianou and Loeffen \cite{KL} investigated the ruin problem of (\ref{eq23}) by establishing a few identities for the  one and two sided exit problems, which are expressed in terms of the scale functions.
They proved that the refracted {\levy} process exists as the unique strong solution to (\ref{eq23}) whenever $X$ is a spectrally negative {\levy} process.

Note that the company with higher reserve has less risk. Therefore, we enrich the model by considering discontinuous diffusion coefficient. Namely, when the reserve process $X_t\ge q$ for a constant $q$, the volatility constant is $\si_1$ which is (usually) lower than the volatility constant $\si_2$ when $X_t<q$. We thus consider the following modified SDE:
\begin{eqnarray}\label{eq22}
dU_t = (\mu_1 \mathbf{1}_{U_t \ge p} + \mu_2 \mathbf{1}_{U_t < p})dt + (\sigma_1 \mathbf{1}_{U_t  \ge q} + \sigma_2 \mathbf{1}_{U_t  < q})dB_t + dJ_t,
\end{eqnarray}
where $p$, $q$, $\sigma_1$ and $\sigma_2$ are positive constants. Suppose that $N$ is a Poisson random measure on $\RR_+\times\RR_-$ with intensity measure $\mu$ on $\RR_-$ satisfying $\int_{\RR_-}u\mu(du)>-\infty$ and
\[J_t=\int^t_0\int_{\RR_-}uN(du ds).\]

Using the results of previous sections, we proceed to proving  pairwise uniqueness for SDE (\ref{eq22}). For simplicity of notation, we denote
\[
b(x) := \mu_1 \mathbf{1}_{x \ge p} + \mu_2 \mathbf{1}_{x < p} ,\quad  \sigma(x) := \sigma_1 \mathbf{1}_{x \ge q} + \sigma_2 \mathbf{1}_{x < q},\quad \text{and}  \quad g(x,u)=u .
\]
It is clear that functions $b(x)$ and  $\sigma(x)$ have at most countably many discontinuous points. It is easy to verify Conditions (3.a) and (4.a,b). To verify the last part of Condition (2) of Theorem 4.2, we define
\[
f(x)=(\sigma_1 - \sigma_2)^2 \mathbf{1}_{\{x > q\}}.
\]
Then
\[
|\sigma(x) - \sigma(y)|^2 \le |f(x) - f(y)|.
\]
Note that $f(x)$ is a bounded and increasing function, The pathwise uniqueness then follows from Theorem 4.2.

\end{document}